\documentclass{article}

\usepackage[T1]{fontenc}
\usepackage{times}
\usepackage{epsfig,graphicx,tabularx,color}
\usepackage[cp850]{inputenc}
\usepackage[english]{babel}
\usepackage{cite}
\usepackage{amsmath,amssymb,amsfonts,amsthm,mathrsfs,comment,mathdots,multirow,tikz}
\usepackage{rotating, longtable, relsize}
\usepackage[f]{esvect}
\usepackage{times}
\usepackage{fancyhdr}
\usepackage{lipsum}

\usetikzlibrary{calc,3d}

\newtheorem{theorem}{Theorem}[section]
\newtheorem{corollary}[theorem]{Corollary}

\newtheorem{definition}[theorem]{Definition}

\theoremstyle{definition}
\newtheorem{remark}[theorem]{Remark}

\def\sphline{\noalign{\vskip3pt}\hline\noalign{\vskip3pt}}

\def\XXint#1#2#3{{\setbox0=\hbox{$#1{#2#3}{\int}$}
     \vcenter{\hbox{$#2#3$}}\kern-.5\wd0}}

\def\ud{{\rm\,d}}
\def\fl{{\rm\,fl}}

\def\D{\mathbb{D}}

\def\N{\mathbb{N}}

\def\R{\mathbb{R}}
\def\Sph{\mathbb{S}}
\def\T{\mathbb{T}}

\def\DD{\mathcal{D}}

\def\II{\mathcal{I}}

\def\MM{\mathcal{M}}

\def\OO{\mathcal{O}}

\def\RR{\mathcal{R}}

\def\VV{\mathcal{V}}

\def\pr(#1){\left({#1}\right)}
\def\br[#1]{\left[{#1}\right]}

\def\abs#1{\left|{#1}\right|}
\def\norm#1{\left\|{#1}\right\|}
\def\conj#1{\overline{#1}}

\def\ii{{\rm i}}
\def\for{\hbox{ for }}

\newcommand{\diag}{\operatorname{diag}}

\begin{document}

\title{Conquering the pre-computation in two-dimensional harmonic polynomial transforms}

\author{%
{\sc Richard Mika\"el Slevinsky\thanks{Corresponding author. Email: Richard.Slevinsky@umanitoba.ca}}\\[2pt]
Department of Mathematics, University of Manitoba, Winnipeg, Canada}

\maketitle

\begin{abstract}
{We describe a skeletonization of the spherical harmonic connection problem that reduces the storage and pre-computation to superoptimal complexities at the cost of increasing the execution time by the modest multiplicative factor of $\mathcal{O}(\log n)$. One advantage of accelerating the spherical harmonic connection problem over accelerating synthesis and analysis is that neighbouring layers (in steps of two) may be expanded in eachother's bases. The proposed skeletonization maximizes this interconnectivity by overlaying a dyadic partitioning on the connection problem. We derive the symmetric-definite banded generalized eigenvalue problem required to accelerate spherical harmonic transforms. We also include a full analysis of the weighted normalized Jacobi connection problem with applications to fast harmonic polynomial transforms on the disk, triangle, rectangle, deltoid, wedge, and any other geometry with a bivariate analogue of Jacobi polynomials.}

{\em Keywords}: harmonic polynomials, connection problem, divide-and-conquer, operator symmetrization.
\end{abstract}

\section{Introduction}

The acceleration of spherical harmonic transforms has received considerable interest. Fast spherical harmonic transforms either accelerate synthesis and analysis, complementing the physical- and momentum-space representations of quantum mechanical operators, or they accelerate the connection problem, converting high-order associated Legendre functions to low-order associated Legendre functions to Fourier series. While many spherical harmonic transforms accelerate synthesis and analysis~\cite{Healy-Rockmore-Kostelec-Moore-9-341-03,Driscoll-Healy-15-202-94,Suda-Takami-71-703-02,Kunis-Potts-161-75-03,Mohlenkamp-5-159-99,Tygert-227-4260-08,Tygert-229-6181-10,Reuter-Ratner-Seideman-131-094108-1-09,Seljebotn-199-12-12,Wedi-Hamrud-Mozdzynski-141-3450-13}, only two methods exist to accelerate the connection problems for spherical harmonics~\cite{Rokhlin-Tygert-27-1903-06,Slevinsky-ACHA-17}. As we discuss, the spherical harmonic connection problem is in fact a problem of {\em interconnection}.

Pre-computations in fast spherical harmonic transforms range from a small multiple of execution times for practical bandwidths~\cite{Slevinsky-ACHA-17}, to the asymptotically optimal complexity of $\OO(n^2\log n)$ in~\cite{Rokhlin-Tygert-27-1903-06,Tygert-227-4260-08}, though the optimal complexity is predicted ``only for absurdly large degrees.'' Practical applications for numerical weather prediction on supercomputers currently use spherical harmonic expansions with $n^2 = \OO(64\times10^6)$ degrees of freedom~\cite{Wedi-Hamrud-Mozdzynski-141-3450-13}; the design and analysis of new fast transforms must take these present and future realities into account.

In this work, we describe a skeletonization of the spherical harmonic connection problem that reduces the pre-computation and storage to the superoptimal complexity of $\OO(n^{\frac{3}{2}}\log n)$ at the cost of increasing the run-time by the modest multiplicative factor of $\OO(\log n)$. The skeletonization takes advantage of the interconnectivity of the spherical harmonic connection problem: it partitions the work asymptotically equally between a data-sparse solution to the neighbouring connection problem and the numerically stable representation of the matrices of connection coefficients as collections of eigenfunctions of a symmetric-definite banded generalized eigenproblem. Through divide-and-conquer eigensolvers accelerated by the Fast Multipole Method~\cite{Greengard-Rokhlin-73-325-87}, optimal pre-computations are realized asymptotically and for practical bandwidths as well. For numerical computing on anything smaller than a supercomputer, the superoptimal complexity of the storage requirement is as important as the pre-computation complexity, for a modern implementation depends even more so on the movement of data than conventional counts of floating-point operations.

Beyond the sphere, harmonic polynomials play an important r\^ole in $L^2$ approximation theory and spectral methods in more exotic two-dimensional geometries. Koornwinder~\cite{Koornwinder-435-75} and Dunkl and Xu~\cite{Dunkl-Xu-14} have released compendia on the bivariate analogues of Jacobi polynomials on the disk, triangle, rectangle, deltoid, among other shapes. Olver and Xu's new work on the orthogonal structure of the wedge and the boundary of the square~\cite{Olver-Xu-17}, with potential applications to singular integral equations~\cite{Slevinsky-Olver-332-290-17}, serves to highlight the genericity of the Jacobi approach when the Laplace--Beltrami operator is separable. We derive the analogous representations of the connection problem to accelerate two-dimensional harmonic polynomial transforms with the same superoptimal complexities for the pre-computation and storage.

\section{Fundamentals}

Let $\mu$ be a positive Borel measure on $D\subset\R^n$. The inner product:
\begin{equation}
\langle f, g \rangle = \int_D \conj{f(x)} g(x)\ud\mu(x),
\end{equation}
where $\conj{f(x)}$ denotes complex conjugation, induces the norm $\norm{f}_2 = \sqrt{\langle f,f\rangle}$ and the associated Hilbert space $L^2(D,\ud\mu(x))$. In case of ambiguity, the notation $\langle f, g\rangle_{{\rm d}\mu}$ is used to distinguish between different measures.

Let $\Sph^2\subset\R^3$ denote the unit $2$-sphere, $\theta\in[0,\pi]$ the co-latitudinal angle, $\varphi\in[0,2\pi)$ the longitudinal angle, and $\ud\Omega = \sin\theta\ud\theta\ud\varphi$ the measure generated by the solid angle $\Omega$ subtended by a spherical cap. Then, any function $f\in L^2(\Sph^2,\ud\Omega)$ may be expanded in spherical harmonics:
\begin{equation}\label{eq:sphericalharmonicexpansion}
f(\theta,\varphi) = \sum_{\ell=0}^{+\infty}\sum_{m=-\ell}^{+\ell} f_\ell^m Y_\ell^m(\theta,\varphi) = \sum_{m=-\infty}^{+\infty}\sum_{\ell=\abs{m}}^{+\infty} f_\ell^m Y_\ell^m(\theta,\varphi),
\end{equation}
where the expansion coefficients are:
\begin{equation}
f_\ell^m = \dfrac{\langle Y_\ell^m, f\rangle}{\langle Y_\ell^m, Y_\ell^m\rangle}.
\end{equation}
Let $\N_0$ denote the non-negative integers. Bandlimiting Eq.~\eqref{eq:sphericalharmonicexpansion} to $\ell\le n\in\N_0$ results in the best degree-$n$ trigonometric polynomial approximation of $f\in L^2(\Sph^2,\ud\Omega)$.

Using the Condon--Shortley phase convention~\cite{Condon-Shortley-51}, orthonormal spherical harmonics are given by:
\begin{equation}\label{eq:sphericalharmonics}
Y_\ell^m(\theta,\varphi) = \dfrac{e^{\ii m\varphi}}{\sqrt{2\pi}} \underbrace{\ii^{m+|m|}\sqrt{(\ell+\tfrac{1}{2})\dfrac{(\ell-m)!}{(\ell+m)!}} P_\ell^m(\cos\theta)}_{\tilde{P}_\ell^m(\cos\theta)},\qquad \ell\in\N_0,\quad -\ell\le m\le \ell,
\end{equation}
where the notation $\tilde{P}_\ell^m$ is used to denote orthonormality for fixed $m$ in the sense of $L^2([-1,1],\ud x)$.

\subsection{The spherical harmonic connection problem}

This subsection is an abbreviated version of the full description by Slevinsky in~\cite{Slevinsky-ACHA-17}.

\begin{definition}
Let $\{\phi_n(x)\}_{n\in\N_0}$ be a family of orthogonal functions with respect to $L^2(\hat{D},\ud\hat{\mu}(x))$ and let $\{\psi_n(x)\}_{n\in\N_0}$ be another family of orthogonal functions with respect to $L^2(D,\ud\mu(x))$. The connection coefficients:
\begin{equation}
c_{\ell,n} = \dfrac{\langle \psi_\ell, \phi_n\rangle_{{\rm d}\mu}}{\langle \psi_\ell, \psi_\ell\rangle_{{\rm d}\mu}},
\end{equation}
allow for the expansion:
\begin{equation}
\phi_n(x) = \sum_{\ell=0}^\infty c_{\ell,n} \psi_\ell(x).
\end{equation}
\end{definition}

\begin{definition}
Let $G_n$ denote the real Givens rotation:
\[
G_n = \begin{pmatrix}
1 & \cdots & 0 & 0 & 0 & \cdots & 0\\
\vdots & \ddots & \vdots & \vdots & \vdots & & \vdots\\
0 & \cdots & c_n & 0 & s_n & \cdots & 0\\
0& \cdots & 0 & 1 & 0 & \cdots & 0\\
0 & \cdots & -s_n & 0 & c_n & \cdots & 0\\
\vdots & & \vdots & \vdots & \vdots & \ddots & \vdots\\
0 & \cdots & 0 & 0 & 0 & \cdots & 1\\
\end{pmatrix},
\]
where the sines $s_n = \sin\theta_n$ and the cosines $c_n = \cos\theta_n$, for some $\theta_n\in[0,2\pi)$, are in the intersections of the $n^{\rm th}$ and $n+2^{\rm nd}$ rows and columns, embedded in the identity of a conformable size. 
\end{definition}

Let $I_{m\times n}$ denote the rectangular identity matrix with ones on the main diagonal and zeros everywhere else.

\begin{theorem}[Slevinsky~\cite{Slevinsky-ACHA-17}]\label{theorem:SS}
The connection coefficients between $\tilde{P}_{n+m+2}^{m+2}(\cos\theta)$ and $\tilde{P}_{\ell+m}^m(\cos\theta)$ are:
\begin{equation}\label{eq:SScoefficients}
c_{\ell,n}^{m} = \left\{\begin{array}{ccc} (2\ell+2m+1)(2m+2)\sqrt{\dfrac{(\ell+2m)!}{(\ell+m+\frac{1}{2})\ell!}\dfrac{(n+m+\frac{5}{2})n!}{(n+2m+4)!}}, & \for & \ell \le n,\quad \ell+n\hbox{ even},\\
-\sqrt{\dfrac{(n+1)(n+2)}{(n+2m+3)(n+2m+4)}}, & \for & \ell = n+2,\\
0, & & otherwise.
\end{array} \right.
\end{equation}
Furthermore, the matrix of connection coefficients $C^{(m)} \in \R^{(n+3)\times (n+1)}$ may be represented via the product of $n$ Givens rotations:
\[
C^{(m)} = G_0^{(m)}G_1^{(m)}\cdots G_{n-2}^{(m)}G_{n-1}^{(m)} I_{(n+3)\times (n+1)},
\]
where the sines and cosines for the Givens rotations are given by:
\begin{equation}\label{eq:GRcoefficients}
s_n^m = \sqrt{\dfrac{(n+1)(n+2)}{(n+2m+3)(n+2m+4)}},\quad{\rm and}\quad c_n^m = \sqrt{\dfrac{(2m+2)(2n+2m+5)}{(n+2m+3)(n+2m+4)}}.
\end{equation}
\end{theorem}

Algorithm~2.5 in~\cite{Slevinsky-ACHA-17} describes a backward stable application of the Givens rotations with the pleasant result that:
\[
\fl(G_n^{(m)}) = G_n^{(m)} + E,\quad{\rm where} \norm{E}_2 = \OO(\epsilon_{\rm mach}),
\]
where $\epsilon_{\rm mach}$ is the unit of least working precision.

\section{Conquering the pre-computation of the connection problem}

The steps required by the spherical harmonic connection problem are illustrated in Figure~\ref{fig:SHT}. At first, an algorithm converts higher-order layers of the spherical harmonics into expansions with orders zero and one. Then, these coefficients are rapidly transformed into their Fourier coefficients. The most expensive part of the transformation is the conversion from high orders to low orders.

\begin{figure}[htbp]
\begin{center}
\begin{tikzpicture}[scale=0.657]
\draw[black, thick]
    (0,0) -- (6,0)
    (1,1) -- (6,1)
    (2,2) -- (6,2)
    (3,3) -- (6,3)
    (5,5) -- (6,5)
    ;
\filldraw[black]
    (0,0) circle (2pt)
    (6,0) circle (2pt)
    (1,1) circle (2pt)
    (6,1) circle (2pt)
    (2,2) circle (2pt)
    (6,2) circle (2pt)
    (3,3) circle (2pt)
    (6,3) circle (2pt)
    (5,5) circle (2pt)
    (6,5) circle (2pt)
    (6,6) circle (2pt)
    ;
\node (ellipsis1) at (4.5,4) {$\iddots$};
\node (ellipsis2) at (5.5,4) {$\vdots$};
\node[anchor=west] (zero) at (6.25,0) {$\tilde{P}_\ell^0$};
\node[anchor=west] (one) at (6.25,1) {$\tilde{P}_\ell^1$};
\node[anchor=west] (two) at (6.25,2) {$\tilde{P}_\ell^2$};
\node[anchor=west] (three) at (6.25,3) {$\tilde{P}_\ell^3$};
\node[anchor=west] (penultimate) at (6.25,5) {$\tilde{P}_\ell^{\ell-1}$};
\node[anchor=west] (ultimate) at (6.25,6) {$\tilde{P}_\ell^\ell$};
\draw[black, thick]
    (9,0) -- (15,0)
    (10,1) -- (15,1)
    (9,2) -- (15,2)
    (10,3) -- (15,3)
    (9,5) -- (15,5)
    (10,6) -- (15,6)
    ;
\filldraw[black]
    (9,0) circle (2pt)
    (15,0) circle (2pt)
    (10,1) circle (2pt)
    (15,1) circle (2pt)
    (9,2) circle (2pt)
    (15,2) circle (2pt)
    (10,3) circle (2pt)
    (15,3) circle (2pt)
    (9,5) circle (2pt)
    (15,5) circle (2pt)
    (10,6) circle (2pt)
    (15,6) circle (2pt)
    ;
\node (ellipsis1) at (11,4) {$\vdots$};
\node (ellipsis2) at (14,4) {$\vdots$};
\node[anchor=west] (zero) at (15.25,0) {$\tilde{P}_\ell^0$};
\node[anchor=west] (one) at (15.25,1) {$\tilde{P}_\ell^1$};
\node[anchor=west] (two) at (15.25,2) {$\tilde{P}_\ell^0$};
\node[anchor=west] (three) at (15.25,3) {$\tilde{P}_\ell^1$};
\node[anchor=west] (penultimate) at (15.25,5) {$\tilde{P}_\ell^0$};
\node[anchor=west] (ultimate) at (15.25,6) {$\tilde{P}_\ell^1$};
\draw[black, thick]
    (18,0) -- (24,0)
    (18,1) -- (23,1)
    (18,2) -- (24,2)
    (18,3) -- (23,3)
    (18,5) -- (24,5)
    (18,6) -- (23,6)
    ;
\filldraw[black]
    (18,0) circle (2pt)
    (24,0) circle (2pt)
    (18,1) circle (2pt)
    (23,1) circle (2pt)
    (18,2) circle (2pt)
    (24,2) circle (2pt)
    (18,3) circle (2pt)
    (23,3) circle (2pt)
    (18,5) circle (2pt)
    (24,5) circle (2pt)
    (18,6) circle (2pt)
    (23,6) circle (2pt)
    ;
\node (ellipsis1) at (19,4) {$\vdots$};
\node (ellipsis2) at (22,4) {$\vdots$};
\node[anchor=west] (zero) at (24.25,0) {$T_\ell$};
\node[anchor=west] (one) at (23.25,1) {$\sin\theta U_\ell$};
\node[anchor=west] (two) at (24.25,2) {$T_\ell$};
\node[anchor=west] (three) at (23.25,3) {$\sin\theta U_\ell$};
\node[anchor=west] (penultimate) at (24.25,5) {$T_\ell$};
\node[anchor=west] (ultimate) at (23.25,6) {$\sin\theta U_\ell$};
\node (firstarrow) at (8.1,3) {$\Longrightarrow$};
\node (secondarrow) at (17.1,3) {$\Longrightarrow$};
\end{tikzpicture}
\caption{The spherical harmonic transform proceeds in two steps. Firstly, normalized associated Legendre functions are converted to normalized associated Legendre functions of order zero and one. Then, these intermediate expressions are re-expanded in trigonometric form.}
\label{fig:SHT}
\end{center}
\end{figure}
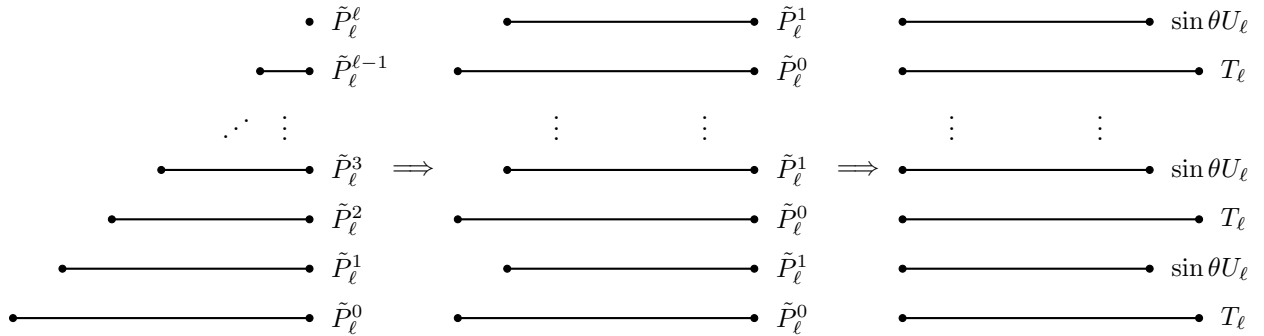

The $\OO(n)$ application of the Givens rotations connecting neighbouring layers allows a partial skeletonization of the pre-computation~\cite{Slevinsky-ACHA-17}. This partial skeletonization reduces the number of layers that require a full butterfly factorization by a constant fraction. However, it still requires {\em all} the Givens rotations to be applied to a pair of dense matrices that are to be compressed, resulting in an $\OO(n^3\log n)$ pre-computation\footnote{though it appears to cost only ten times the execution for presented data.}.

\begin{figure}[htbp]
\begin{center}
\begin{tikzpicture}[scale=12]
\draw[stealth-stealth, thick] (-0.0625,0) -- (1.0625,0);
\foreach \x in {0.5}
    \draw[bend right = 75, line width = 2.0] (\x,0) edge[-stealth] (\x-0.5,0);
\foreach \x/\xtext in { 0/0, 1/n} {
  \draw[shift={(\x,0)}] (0,0.02) -- (0,-0.02);
  \draw(\x,0) node[below,scale=2.5] {$\xtext$};
  }
\foreach \x/\xtext in {0.5/\frac{n}{2}} {
  \draw[shift={(\x,0)}] (0,0.01) -- (0,-0.01);
  \draw(\x,0) node[below,scale=1.75] {$\xtext$};
  }
    \pgfmathparse{2}\edef\i{\pgfmathresult}
    \pgfmathparse{2^\i}\edef\jmax{\pgfmathresult}
    \pgfmathparse{int(2^\i)}\edef\xd{\pgfmathresult}
    \pgfmathparse{(0.5^(\i))}\edef\x{\pgfmathresult}
    \pgfmathparse{4.0/((\i+1))}\edef\scl{\pgfmathresult}
    \draw[shift={(\x,0)}] (0,0.005*\scl) -- (0,-0.005*\scl);
    \draw(\x,0) node[below,scale=\scl] {$\frac{n}{\xd}$};
    \foreach \j in {3} {
        \pgfmathparse{\j*(0.5^(\i))}\edef\x{\pgfmathresult};
      \draw[shift={(\x,0)}] (0,0.005*\scl) -- (0,-0.005*\scl);
        \draw(\x,0) node[below,scale=\scl] {$\frac{\j n}{\xd}$};
        }
\foreach \i in {3,...,4} {
    \pgfmathparse{2^\i}\edef\jmax{\pgfmathresult}
    \pgfmathparse{int(2^\i)}\edef\xd{\pgfmathresult}
    \pgfmathparse{(0.5^(\i))}\edef\x{\pgfmathresult}
    \pgfmathparse{4.0/((\i+1))}\edef\scl{\pgfmathresult}
    \draw[shift={(\x,0)}] (0,0.005*\scl) -- (0,-0.005*\scl);
    \draw(\x,0) node[below,scale=\scl] {$\frac{n}{\xd}$};
    \foreach \j in {3,5,...,\jmax} {
        \pgfmathparse{\j*(0.5^(\i))}\edef\x{\pgfmathresult};
      \draw[shift={(\x,0)}] (0,0.005*\scl) -- (0,-0.005*\scl);
        \draw(\x,0) node[below,scale=\scl] {$\frac{\j n}{\xd}$};
        }
    }
    \pgfmathparse{5}\edef\i{\pgfmathresult}
    \pgfmathparse{2^\i}\edef\jmax{\pgfmathresult}
    \pgfmathparse{int(2^\i)}\edef\xd{\pgfmathresult}
    \pgfmathparse{(0.5^(\i))}\edef\x{\pgfmathresult}
    \pgfmathparse{4.0/((\i+1))}\edef\scl{\pgfmathresult}
    \draw(\x,0) node[below,scale=\scl] {$\stackrel{~}{\cdots}$};
    \foreach \j in {3,5,...,\jmax} {
        \pgfmathparse{\j*(0.5^(\i))}\edef\x{\pgfmathresult};
        \draw(\x,0) node[below,scale=\scl] {$\stackrel{~}{\cdots}$};
        }
\foreach \i in {1,...,5}
    \pgfmathparse{0.5^(\i+1)}\edef\xs{\pgfmathresult}
    \pgfmathparse{3*(0.5^(\i+1))}\edef\xstep{\pgfmathresult}
    \pgfmathparse{1-0.5^(\i+1)}\edef\xf{\pgfmathresult}
    \pgfmathparse{75/((\i+1)^0.125)}\edef\br{\pgfmathresult}
    \pgfmathparse{2.0/((\i+1))}\edef\lw{\pgfmathresult}
    \foreach \x in {\xs,\xstep,...,\xf}
        \draw[bend right = \br, line width = \lw] (\x,0) edge[-stealth] (\x-\xs,0);
\end{tikzpicture}
\caption{The complete skeletonization of the pre-computation via a dyadic partitioning. Each arrow indicates the construction of an accelerated spectral decomposition between expansions in the higher-order layer and the lower-order layer.}
\label{fig:Skeleton}
\end{center}
\end{figure}
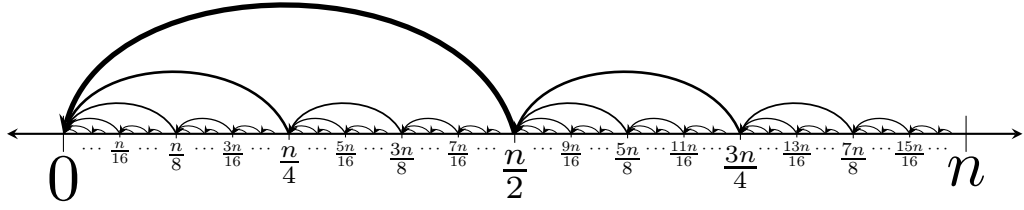

In Figure~\ref{fig:Skeleton}, we propose a {\em complete} skeletonization by dyadically partitioning the pre-computation; the arrows denote the construction of an accelerated spectral decomposition of the connection problem.

In Table~\ref{table:Costs}, the computational complexities are tabulated to compute Givens and spectral representations of the connection problems. The costs for the Givens rotations are tabulated assuming they are used to convert all spherical harmonic layers $\{m+2, m+4,\ldots, \mu\}$ down to layer $m$, resulting in the total execution time of order $\OO((\mu-m)^2n)$. Once neighbouring layers are converted to batches of representations in the same bases, fast spectral decompositions are used to finish the transformation to expansions with orders zero and one by traversing the binary tree.

Any given layer sees at most $\log_2\lceil\frac{n}{\mu-m}\rceil$ spectral decompositions on its transition to an expansion in the zeroth or first order, depending on the even-odd symmetry. To overlay a dyadic partitioning, the total number of spectral decompositions that require pre-computation is:
\[
\#\{\hbox{Spectral Decompositions}\} = \sum_{i=1}^{\log_2\lceil\frac{n}{\mu-m}\rceil} 2^{i-1} = \OO\left(\frac{n}{\mu-m}\right).
\]

\begin{table}
\caption{The costs of a dyadically partitioned spherical harmonic transform.}
\begin{center}
\begin{tabular}{cccc}
\sphline
Cost & Pre-computation & Execution & Storage\\
\sphline
All Givens rotations & $\OO(n^2)$ & $\OO((\mu-m)^2n)$ & $\OO(n^2)$\\
A spectral decomposition & $\Pi(n)$ & $\Phi(n)$ & $\Sigma(n)$\\
\sphline
Summary & $\OO(n^2) + \OO(\tfrac{n}{\mu-m})\Pi(n)$ & $\OO((\mu-m)^2n) + n\Phi(n)\log_2\lceil\tfrac{n}{\mu-m}\rceil$ & $\OO(n^2)+\OO(\tfrac{n}{\mu-m})\Sigma(n)$\\
\sphline
If $\mu-m = \OO(\sqrt{n})$ & $\OO(n^2) + \OO(\sqrt{n})\Pi(n)$ & $\OO(n^2) + n\Phi(n)\OO(\log_2 n)$ & $\OO(n^2)+\OO(\sqrt{n})\Sigma(n)$\\
\sphline
\end{tabular}
\end{center}
\label{table:Costs}
\end{table}%

In Table~\ref{table:Costs}, a fast spectral decomposition costs $\Pi(n)$ to pre-compute, $\Phi(n)$ to execute, and $\Sigma(n)$ to store. Under the assumption that $\mu-m = \OO(\sqrt{n})$, then the execution time is nearly asymptotically optimal if $\Phi(n) = \OO(n\log n)$. Furthermore, if the Givens rotations are computed on-the-fly, then pre-computation and storage are superoptimal in the sense that they require fewer floating-point operations or units of memory than the number of degrees of freedom in the spherical harmonic expansion. This holds so long as $\Pi(n) = \OO(n\log n)$ and $\Sigma(n) = \OO(n\log n)$ as well.

Regarding implementation, the dyadic partitioning creates batches of coefficients whose transforms may take full advantage of level-III BLAS~\cite{OpenBLAS} compiler optimizations. Therefore, while the execution adds a logarithmic factor accounting for the number of levels, we anticipate any algorithm that skeletonizes completely to exhibit superior performance in practice.

\subsection{Fast harmonic eigenfunction transforms}

Normalized associated Legendre functions of degree $\ell$ and order $m$ are eigenfunctions of the linear differential equation:
\begin{equation}\label{eq:SLP_ALF}
-(1-x^2)\dfrac{{\rm d}}{{\rm d}x}\left[(1-x^2)\dfrac{{\rm d}}{{\rm d}x}\tilde{P}_\ell^m(x)\right] + m^2\tilde{P}_\ell^m(x) = \ell(\ell+1)(1-x^2)\tilde{P}_\ell^m(x), \quad \abs{m} \le \ell.
\end{equation}
If we expand normalized associated Legendre functions of order $\mu$ in the basis $\tilde{P}_\ell^m(x)$:
\[
\tilde{P}_n^\mu(x) = \sum_{\ell} c_{\ell,n}^{m,\mu} \tilde{P}_\ell^m(x),
\]
then we may rewrite the Sturm--Liouville eigenproblem in Eq.~\eqref{eq:SLP_ALF} more conveniently. If we symbolically let $\MM$ denote the operation of multiplying the basis by $1-x^2$, and if we let $\DD$ denote the scaling of the basis by $\ell(\ell+1)$ for $\abs{m}\le\ell$, we have the generalized eigenvalue problem:
\begin{equation}\label{eq:SHT}
\left(\MM \DD + (\mu^2-m^2)\II\right)u = \lambda \MM u.
\end{equation}
Due to the recurrence relation:
\[
x\tilde{P}_{\ell+m}^m = \sqrt{\dfrac{(\ell+1)(\ell+2m+1)}{(2\ell+2m+1)(2\ell+2m+3)}}\tilde{P}_{\ell+m+1}^m + \sqrt{\dfrac{\ell(\ell+2m)}{(2\ell+2m-1)(2\ell+2m+1)}}\tilde{P}_{\ell+m-1}^m,
\]
multiplication by $1-x^2$ is represented as the symmetric and positive-definite operator~\cite[see also Lemma~2.12]{Tygert-227-4260-08}:
\[
\MM = \begin{pmatrix} a_1 & 0 & b_1\\ 0 & a_2 & & b_2\\ b_1 & & a_3 & & b_3\\ & \ddots & & \ddots & & \ddots\\ & & b_{n-2} & & a_n & & b_n\\ & & & \ddots & & \ddots & & \ddots\end{pmatrix},
\]
where:
\[
a_n = \dfrac{2(n^2+2mn+2m^2-n-m-1)}{(2n+2m-3)(2n+2m+1)},\quad{\rm and}\quad b_n = -\sqrt{\dfrac{n(n+1)(n+2m)(n+2m+1)}{(2n+2m-1)(2n+2m+1)^2(2n+2m+3)}}.
\]
Unfortunately, the symmetry in $\MM$ is destroyed by the column scaling in $\DD$. Rokhlin and Tygert~\cite{Rokhlin-Tygert-27-1903-06} realized that the generalized eigenvalue problem\footnote{In their formul\ae, $m=1,2$, whereas {\em all integer values} of $m\ne0$ are required to conquer the pre-computation.} is symmetrized when multiplied by $\MM^{-1}$ from the left, resulting in:
\[
\left(\DD + (\mu^2-m^2)\MM^{-1} \right)u = \lambda u.
\]
This symmetric diagonal-plus-semiseparable regular eigenvalue problem is the starting point for their fast spherical harmonic transform~\cite{Rokhlin-Tygert-27-1903-06}: the FMM~\cite{Greengard-Rokhlin-73-325-87} accelerates Chandrasekaran and Gu's divide-and-conquer algorithm for symmetric diagonal-plus-semiseparable eigenvalue problems~\cite{Chandrasekaran-Gu-96-723-04}. To conquer the pre-computation, the inverse multiplication operator would be required for the interconnection between layers.
\begin{theorem}\label{theorem:MInverse}
The inverse of the multiplication operator is given by the symmetric semiseparable operator:
\[
[\MM^{-1}]_{\ell,n} = 
\left\{\begin{array}{ccc} \dfrac{\sqrt{(2\ell+2m+1)(2n+2m+1)}}{2m}\sqrt{\dfrac{(\ell+2m)!}{\ell!}\dfrac{n!}{(n+2m)!}}, & \for & \ell \le n,\quad \ell+n\hbox{ even},\\
\dfrac{\sqrt{(2\ell+2m+1)(2n+2m+1)}}{2m}\sqrt{\dfrac{\ell!}{(\ell+2m)!}\dfrac{(n+2m)!}{n!}}, & \for & \ell > n,\quad \ell+n\hbox{ even},\\
0, & & otherwise.
\end{array} \right.
\]
\end{theorem}
\begin{proof}
Formally, the inverse of the multiplication operator has entries determined by the integrals:
\[
[\MM^{-1}]_{\ell,n} = \int_{-1}^1 \dfrac{\tilde{P}_{\ell+m}^m(x)\tilde{P}_{n+m}^m(x)}{1-x^2}\ud x, \quad{\rm for}\quad \ell,n \ge 0.
\]
Using the recurrence:
\[
\dfrac{1}{\sqrt{1-x^2}}\tilde{P}_{\ell+m}^m = \frac{1}{2m}\sqrt{\dfrac{\ell+m+\frac{1}{2}}{\ell+m+\frac{3}{2}}}\left(\sqrt{(\ell+1)(\ell+2)}\tilde{P}_{\ell+m+1}^{m-1} + \sqrt{(\ell+2m+1)(\ell+2m+2)}\tilde{P}_{\ell+m+1}^{m+1}\right),
\]
we may write:
\begin{align*}
[\MM^{-1}]_{\ell,n} & = \dfrac{1}{(2m)^2}\sqrt{\dfrac{(\ell+m+\frac{1}{2})(n+m+\frac{1}{2})}{(\ell+m+\frac{3}{2})(n+m+\frac{3}{2})}}\\
& \times\Bigg( \sqrt{(\ell+1)(\ell+2)(n+1)(n+2)}\int_{-1}^1\tilde{P}_{\ell+m+1}^{m-1}(x)\tilde{P}_{n+m+1}^{m-1}(x)\ud x\\
& \quad+ \sqrt{(\ell+1)(\ell+2)(n+2m+1)(n+2m+2)}\int_{-1}^1\tilde{P}_{\ell+m+1}^{m-1}(x)\tilde{P}_{n+m+1}^{m+1}(x)\ud x\\
& \quad+ \sqrt{(\ell+2m+1)(\ell+2m+2)(n+1)(n+2)}\int_{-1}^1\tilde{P}_{\ell+m+1}^{m+1}(x)\tilde{P}_{n+m+1}^{m-1}(x)\ud x\\
& \quad+ \sqrt{(\ell+2m+1)(\ell+2m+2)(n+2m+1)(n+2m+2)}\int_{-1}^1\tilde{P}_{\ell+m+1}^{m+1}(x)\tilde{P}_{n+m+1}^{m+1}(x)\ud x\Bigg).
\end{align*}
The first and last integrals are given by orthonormality, while the second and third integrals may be evaluated by Theorem~\ref{theorem:SS}.
\end{proof}

While Theorem~\ref{theorem:MInverse} shows that it is possible, at least theoretically, to conquer the pre-computation with the inverse multiplication operator, it demands an unconventional floating-point number system with an exceptionally long mantissa to account for the severe factorial growth and decay in the semiseparable factors. To remedy this issue, we develop an alternative approach to the fast spectral decomposition that is more favourable to IEEE floating-point arithmetic.

If we take the Cholesky factorization of $\MM = \RR^\top\RR$, if we consider new generalized eigenvectors $v = \RR^{-\top} u$, and if we multiply Eq.~\eqref{eq:SHT} by $\RR^{-\top}$ from the left, then we have transformed the problem to:
\[
\left( \RR \DD \RR^\top + (\mu^2-m^2)\II\right) v = \lambda \RR\RR^\top v.
\]
This generalized eigenvalue problem is symmetric-definite and pentadiagonal\footnote{It is in fact tridiagonal if the even-odd symmetry is highlighted by a perfect shuffle.}, allowing for a fast spectral decomposition~\cite{Borges-Gragg-11-93,Gu-Eisenstat-16-172-95}. Furthermore, the generalized eigenvectors $\VV$ are $\RR\RR^\top$-orthogonal in the sense that:
\[
\VV^\top \RR \RR^\top \VV = \II.
\]

In Appendix~\ref{appendix:SHT}, we derive explicit formul\ae~for the Cholesky factor $\RR$ and the products $\RR\RR^\top$ and $\RR\DD\RR^\top$.

The compromise of reformulating the eigenvalue problem in a generalized sense is the insurgence of ill-conditioning from the Cholesky factor which appears to be of order $\OO(n)$. An alternative strategy would be to explore fast eigensolvers for symmetric diagonal-plus-generator-representable matrices, since the entries of $\MM^{-1}$ may be computed to high relative accuracy by a product of Givens rotations through Eq.~\eqref{eq:GRcoefficients}.

\begin{remark}
The reformulation as a symmetric-definite banded generalized eigenvalue problem appears to seek the second family of linearly independent solutions to the Sturm--Liouville problem when $\mu = m$; thus, it appears not to recover eigenfunctions expressible as the identity operator. This is likely due to the omission of boundary conditions. When $\mu-m$ is a positive even integer, the reformulation appears to seek the eigenfunctions whose expansions terminate, as desired.
\end{remark}

\section{Divide-and-conquer symmetric tridiagonal (generalized) eigenvalue solvers}

Let $T\in\R^{n\times n}$ be a symmetric tridiagonal matrix. The spectral decomposition:
\[
T = Q\Lambda Q^\top,
\]
may be performed in $\OO(n^2)$ operations by divide-and-conquer algorithms~\cite{Cuppen-36-177-81,Gu-Eisenstat-16-172-95}, bisection, and the method of multiple relatively robust representations~\cite{Dhillon-Parlett-387-1-04,Dhillon-Parlett-Vomel-32-533-06}. When divide-and-conquer algorithms are accelerated by the Fast Multipole Method (FMM)~\cite{Greengard-Rokhlin-73-325-87}, the eigenvalues are obtained in $\OO(n\log n)$ operations, and a structured representation of the eigenvectors allows for the matrix vector products $Qx$ and $Q^\top x$ in $\OO(n\log n)$ operations as well.

Let $S\in\R^{n\times n}$ be a symmetric positive-definite tridiagonal matrix. The generalized eigenvalues and eigenvectors diagonalize the pencil $(T,S)$:
\[
V^\top (T-\lambda S)V = \Lambda - \lambda I.
\]
A divide-and-conquer algorithm is derived by Borges and Gragg~\cite{Borges-Gragg-11-93} that is technically similar to the symmetric tridiagonal divide-and-conquer algorithm of Gu and Eisenstat~\cite{Gu-Eisenstat-16-172-95}. The main addition is the use of a sparse Cholesky factorization of a symmetric positive-definite arrowhead matrix to relate a symmetric-definite arrowhead generalized eigenvalue problem to a regular one.

\subsection{Divide}

If the tridiagonal matrix $T$ is partitioned as:
\[
T = \begin{pmatrix} T_1 & a\\ a^\top & c & b^\top\\ & b & T_2\end{pmatrix},
\]
where the dimensions of $T_1$ and $T_2$ are nearly half of those of $T$, then the spectral decompositions of the symmetric tridiagonal matrices $T_1 = Q_1\Lambda_1Q_1^\top$ and $T_2 = Q_2\Lambda_2Q_2^\top$ allow for a similarity transformation to the symmetric arrowhead matrix:
\begin{equation}\label{eq:SymTridivide}
\begin{pmatrix} Q_1\\ & & 1\\& Q_2\end{pmatrix}^\top \begin{pmatrix} T_1 & a\\ a^\top & c & b^\top\\ & b & T_2\end{pmatrix} \begin{pmatrix} Q_1\\ & & 1\\& Q_2\end{pmatrix} = \begin{pmatrix} \Lambda_1 & & Q_1^\top a\\ & \Lambda_2 & Q_2^\top b\\ a^\top Q_1 & b^\top Q_2 & c\end{pmatrix}.
\end{equation}
The same division of the symmetric tridiagonal eigenproblems $T_1$ and $T_2$ may be used recursively until the dimensions are sufficiently small for conventional eigensolvers to be competitive.

\subsection{Conquer}

Consider the symmetric arrowhead matrix $A\in\R^{n\times n}$ which could arise from Eq.~\eqref{eq:SymTridivide}:
\[
A = \begin{pmatrix} D & b\\ b^\top & c\end{pmatrix} = \begin{pmatrix} a_1 & & & & b_1\\ & a_2 & & & b_2\\ & & \ddots & & \vdots\\ & & & a_{n-1} & b_{n-1}\\ b_1 & b_2 & \cdots & b_{n-1} & c\end{pmatrix},
\]
where the diagonal elements are nondecreasing $a_1 \le a_2 \le \cdots \le a_{n-1}$. If we perform the $LU$ factorization of the matrix $A-\lambda I$, then:
\[
A - \lambda I = \begin{pmatrix} I & 0\\ b^\top (D-\lambda I)^{-1} & 1\end{pmatrix} \begin{pmatrix} D-\lambda I & b\\ 0^\top & -f(\lambda)\end{pmatrix},
\]
where $f$ is the rational Pick function given by~\cite{Borges-Gragg-11-93}:
\[
f(\lambda) = \lambda - c + \sum_{i=1}^{n-1} \frac{b_i^2}{a_i-\lambda}.
\]
Based on the graph of $f$, the roots of $f$ interlace the elements of the arrowhead's shaft:
\[
\lambda_1 \le a_1 \le \lambda_2 \le a_2 \le \cdots \le a_{n-1} \le \lambda_n.
\]
The roots of the Pick function correspond to the eigenvalues of $A$, and Borges and Gragg devised a cubically convergent algorithm~\cite{Borges-Gragg-11-93} by fitting another rational function:
\[
\phi_j(\lambda) = \alpha + \frac{\beta}{a_{j-1} - \lambda} + \frac{\gamma}{a_j-\lambda},\quad{\rm for}\quad 2\le j \le n-1,
\]
to the Pick function and its first and second derivatives at the current estimate for the root. Modified versions of the rational fitting are used for the first and last roots that also guarantee convergence. Note that the Pick function $f(\lambda)$ and its higher order derivatives can be rapidly evaluated at $m$ points in $\OO(m+n)$ floating-point operations by the FMM; evaluation of the Pick function at $m$ points requires a matrix-vector product with the Cauchy matrix, a procedure for which FMM acceleration is particularly well-suited.

Next, a nearby symmetric arrowhead matrix is constructed from the computed eigenvalues and the shaft of the original matrix.
\begin{theorem}[Boley and Golub~\cite{Boley-Golub-23-630-77}]\label{theorem:BoleyAndGolub}
Given a set of numbers $\{\hat{\lambda}_i\}_{i=1}^n$ and a diagonal matrix $D = \diag(a_1,\ldots,a_{n-1})$ satisfying the interlacing property:
\[
\hat{\lambda}_1 < a_1 < \hat{\lambda}_2 < \cdots < a_{n-1} < \hat{\lambda}_n,
\]
there exists a symmetric arrowhead matrix:
\[
\hat{A} = \begin{pmatrix} D & \hat{b}\\ \hat{b}^\top & \hat{c}\end{pmatrix},
\]
whose eigenvalues are $\{\hat{\lambda}_i\}_{i=1}^n$. The vector $\hat{b} = (\hat{b}_1,\ldots,\hat{b}_{n-1})^\top$ is given by:
\begin{equation}\label{eq:spike}
\hat{b}_i = {\rm sign}(b_i)\sqrt{(a_i-\hat{\lambda}_1)(\hat{\lambda}_n-a_i)\prod_{j=1}^{i-1}\dfrac{\hat{\lambda}_{j+1}-a_i}{a_j-a_i}\prod_{j=i+1}^{n-1}\dfrac{\hat{\lambda}_{j}-a_i}{a_{j}-a_i}},
\end{equation}
and the scalar $\hat{c}$ is:
\begin{equation}\label{eq:cap}
\hat{c} = \hat{\lambda}_n + \sum_{i=1}^{n-1}(\hat{\lambda}_i - a_i).
\end{equation}
\end{theorem}
Theorem~\ref{theorem:BoleyAndGolub} allows us to compute a new symmetric arrowhead matrix whose eigenvalues are known {\em exactly}. The modified matrix is nearby the original matrix in $2$-norm so long as $\hat{b}$ and $\hat{c}$ are close to $b$ and $c$, respectively:
\[
\norm{A-\hat{A}}_2 \le \norm{b-\hat{b}}_2 + \abs{c-\hat{c}}.
\]
Furthermore, each difference, each product, and each ratio in~\eqref{eq:spike} and each difference in~\eqref{eq:cap} may be computed to high relative accuracy. When the products are rewritten as exponentials of sums of logarithms, FMM acceleration is unlocked.

Finally, a $2$-normalized eigenvector of the nearby symmetric arrowhead matrix is given by:
\begin{equation}\label{eq:symarroweigenvector}
\hat{q}_i = \left(\dfrac{\hat{b}_1}{\hat{\lambda}_i-a_1},\ldots,\dfrac{\hat{b}_{n-1}}{\hat{\lambda}_i-a_{n-1}},1\right)^\top\Bigg/ \sqrt{1+\sum_{j=1}^{n-1}\dfrac{\hat{b}_j^2}{(\hat{\lambda}_i - a_j)^2}},
\end{equation}
another expression that is conveniently accelerated by the FMM. The ingenuity of this approach, by Gu and Eisenstat~\cite{Gu-Eisenstat-16-172-95}, is that the collection of all orthonormal eigenvectors is {\em numerically orthogonal} to the working precision, whereas the individual ratios in Eq.~\eqref{eq:symarroweigenvector} cannot be computed to high relative accuracy when the perturbed spike $\hat{b}$ is replaced with the original spike $b$.

\section{Harmonic polynomials based on bivariate analogues of Jacobi polynomials}

Jacobi polynomials are the orthogonal polynomials with respect to $L^2([-1,1], (1-x)^\alpha(1+x)^\beta\ud x)$ where $\alpha,\beta>-1$. They satisfy the symmetry relation~\cite[\S 22.1]{Abramowitz-Stegun-65}:
\begin{equation}\label{eq:Pnsym}
P_n^{(\alpha,\beta)}(-x) = (-1)^n P_n^{(\beta,\alpha)}(x),
\end{equation}
and the two recurrence relations~\cite[\S 22.1]{Abramowitz-Stegun-65}:
\begin{align}
(2n+\alpha+\beta+1)P_n^{(\alpha,\beta)}(x) & = (n+\alpha+\beta+1)P_n^{(\alpha+1,\beta)}(x) - (n+\beta)P_{n-1}^{(\alpha+1,\beta)}(x),\\
(2n+\alpha+\beta+2)(1-x)P_n^{(\alpha+1,\beta)}(x) & = 2(n+\alpha+1)P_n^{(\alpha,\beta)}(x) - 2(n+1)P_{n+1}^{(\alpha,\beta)}(x).
\end{align}
The orthonormality constants~\cite[\S 22.1]{Abramowitz-Stegun-65}:
\begin{equation}
\langle P_n^{(\alpha,\beta)}, P_n^{(\alpha,\beta)} \rangle = \dfrac{2^{\alpha+\beta+1}\Gamma(n+\alpha+1)\Gamma(n+\beta+1)}{(2n+\alpha+\beta+1)\Gamma(n+\alpha+\beta+1)n!}.
\end{equation}
allow us to define orthonormalized Jacobi polynomials:
\begin{equation}
\tilde{P}_n^{(\alpha,\beta)}(x) := \sqrt{\dfrac{(2n+\alpha+\beta+1)\Gamma(n+\alpha+\beta+1)n!}{2^{\alpha+\beta+1}\Gamma(n+\alpha+1)\Gamma(n+\beta+1)}}P_n^{(\alpha,\beta)}(x).
\end{equation}

Several classes of two-dimensional harmonic polynomials are based on generalizations of Jacobi polynomials~\cite{Koornwinder-435-75,Dunkl-Xu-14}. In particular, Jacobi polynomials generalize to the disk, the triangle, the rectangle, the deltoid, and recently the wedge and the boundary of the square~\cite{Olver-Xu-17}, among other shapes.

\subsection{The disk}

Let $\D^2 = \{(r,\theta)\in[0,1]\times[0,2\pi)\} \subset \R^2$ denote the unit disk, $r\in[0,1]$ the radius, and $\theta\in[0,2\pi)$ the polar angle measured counterclockwise from the positive $x$-axis. We consider orthogonal polynomials on the Hilbert space $L^2(\D^2, r\ud r\ud\theta)$. On the disk, any function $f\in L^2(\D^2, r\ud r\ud\theta)$ may be expanded in disk harmonics, or so-called Zernike polynomials~\cite{Zernike-1-689-34}:
\begin{equation}\label{eq:diskharmonicexpansion}
f(r,\theta) = \sum_{\ell=0}^{+\infty}\sum_{m=-\ell,2}^{+\ell} \dfrac{\langle Z_\ell^m, f\rangle}{\langle Z_\ell^m, Z_\ell^m\rangle} Z_\ell^m(r,\theta) = \sum_{m=-\infty}^{+\infty}\sum_{\ell=\abs{m},2}^{+\infty} \dfrac{\langle Z_\ell^m, f\rangle}{\langle Z_\ell^m, Z_\ell^m\rangle} Z_\ell^m(r,\theta),
\end{equation}
where the inner summation indices runs in steps of two.

The orthonormalized Zernike polynomials are bivariate analogues of Jacobi polynomials:
\begin{equation}
Z_\ell^m(r,\theta) = \dfrac{e^{\ii m\theta}}{\sqrt{2\pi}} \sqrt{2\ell+2} r^{\abs{m}} P_{\frac{\ell-\abs{m}}{2}}^{(0,\abs{m})}(2r^2-1),\qquad \ell\in\N_0,\quad m\in\{-\ell,-\ell+2,\ldots,\ell-2,\ell\}.
\end{equation}

Zernike polynomials have numerous applications in optics~\cite{Noll-66-207-76} and spectral methods for partial differential equations (PDEs) on the unit disk~\cite{Vasil-et-al-325-53-16}. Recently, the Fourier series of Zernike polynomials are derived in closed-form~\cite{Janssen-6-11028-1-11}. It is worth investigating whether or not fast transforms are producible based on the analytical expressions. Nevertheless, Zernike polynomials fall into Koornwinder's classification of bivariate analogues of Jacobi polynomials whose transforms may be accelerated in this framework.

\subsection{The triangle}

Without loss of generality, we work with a unit right triangle since polynomial structure and orthogonality are preserved under affine transforms~\cite{Olver-Xu-17}. Let $\T^2 = \{(x,y)\in\R^2 : 0\le x,y,x+y\le 1\}\subset\R^2$ denote the unit right triangle, let $w(x,y) = 2^{\alpha+2\beta+2\gamma+2}x^\alpha y^\beta (1-x-y)^\gamma$ and let $\ud\mu(x,y) = w(x,y) \ud x\ud y$. We consider orthogonal polynomials on the Hilbert space $L^2(\T^2, \ud\mu(x,y))$. The bivariate analogues of Jacobi polynomials are~\cite{Koornwinder-435-75,Dunkl-Xu-14}:
\begin{equation}
\tilde{P}_{\ell,m}^{(\alpha,\beta,\gamma)}(x,y) = (2(1-x))^m \tilde{P}_{\ell-m}^{(2m+\beta+\gamma+1,\alpha)}(2x-1) \tilde{P}_m^{(\gamma,\beta)}\left(\frac{2y}{1-x}-1\right),
\end{equation}
where the tilde implies that the univariate Jacobi polynomials are normalized such that:
\begin{equation}
\int_{\T^2} \tilde{P}_{\ell,m}^{(\alpha,\beta,\gamma)}(x,y)\tilde{P}_{\ell',m'}^{(\alpha,\beta,\gamma)}(x,y) w(x,y)\ud x\ud y = \delta_{\ell,\ell'}\delta_{m,m'}.
\end{equation}
On the triangle, any function $f\in L^2(\T^2,\ud\mu(x,y))$ may be expanded in triangular harmonics:
\[
f(x,y) = \sum_{\ell=0}^{+\infty}\sum_{m = 0}^\ell \langle \tilde{P}_{\ell,m}^{(\alpha,\beta,\gamma)}, f\rangle \tilde{P}_{\ell,m}^{(\alpha,\beta,\gamma)}(x,y).
\]
It is anticipated that triangular harmonics play a r\^ole in the spectral element method where a degree-$1$ refinement is taken in the geometry, and the PDEs are solved in infinite dimensions on each triangle.

\subsection{The Jacobi connection problem}

We begin by deriving a Givens rotation representation of the connection problem between neighbouring weighted Jacobi polynomials, similar in spirit to Theorem~\ref{theorem:SS}. Once weighted Jacobi expansions with high parameters are converted to equivalent representations with low parameters, a Chebyshev--Jacobi~\cite{Slevinsky-17} or Jacobi--Jacobi~\cite{Townsend-Webb-Olver-17} transform may be used to convert all expansions to a more convenient representation for rapid evaluation on a grid.

\begin{theorem}[Andrews, Askey, and Roy~\cite{Andrews-Askey-Roy-98}]
\begin{equation}\label{eq:Jacobisingleparamcnxn}
P_n^{(\gamma,\beta)}(x) = \dfrac{(\beta+1)_n}{(\alpha+\beta+2)_n} \sum_{k=0}^n \dfrac{(\gamma-\alpha)_{n-k}(\alpha+\beta+1)_k(\alpha+\beta+2k+1)(\beta+\gamma+n+1)_k}{(n-k)!(\beta+1)_k(\alpha+\beta+1)(\alpha+\beta+n+2)_k}P_k^{(\alpha,\beta)}(x).
\end{equation}
\end{theorem}

\begin{definition}
Let $G_n$ denote the Givens rotation:
\[
G_n = \begin{pmatrix}
1 & \cdots & 0 & 0 & \cdots & 0\\
\vdots & \ddots & \vdots & \vdots & & \vdots\\
0 & \cdots & c_n & s_n & \cdots & 0\\
0 & \cdots & -s_n & c_n & \cdots & 0\\
\vdots & & \vdots & \vdots & \ddots & \vdots\\
0 & \cdots & 0 & 0 & \cdots & 1\\
\end{pmatrix},
\]
where the sines and the cosines are in the intersections of the $n^{\rm th}$ and $n+1^{\rm st}$ rows and columns, embedded in the identity of a conformable size.
\end{definition}

\begin{theorem}\label{theorem:Jacobi}
The connection coefficients between $(1-x)\tilde{P}_n^{(\alpha+2,\beta)}(x)$ and $\tilde{P}_{\ell}^{(\alpha,\beta)}(x)$ are:
\begin{equation}\label{eq:Jacobicoefficients}
c_{\ell,n}^{(\alpha,\beta)} = \left\{\begin{array}{ccc} (\alpha+1)\sqrt{\dfrac{\dfrac{(2\ell+\alpha+\beta+1)\Gamma(\ell+\alpha+\beta+1)\Gamma(\ell+\alpha+1)}{\Gamma(\ell+\beta+1)\Gamma(\ell+1)}}{\dfrac{\Gamma(n+\alpha+\beta+3)\Gamma(n+\alpha+3)}{(2n+\alpha+\beta+3)\Gamma(n+\beta+1)\Gamma(n+1)}}}, & \for & \ell \le n,\\
-\sqrt{\dfrac{(n+1)(n+\beta+1)}{(n+\alpha+2)(n+\alpha+\beta+2)}}, & \for & \ell = n+1,\\
0, & & otherwise.
\end{array} \right.
\end{equation}
Furthermore, the matrix of connection coefficients $C^{(\alpha,\beta)} \in \R^{(n+2)\times (n+1)}$ may be represented via the product of $n$ Givens rotations:
\[
C^{(\alpha,\beta)} = G_0^{(\alpha,\beta)}G_1^{(\alpha,\beta)}\cdots G_{n-2}^{(\alpha,\beta)}G_{n-1}^{(\alpha,\beta)} I_{(n+2)\times (n+1)},
\]
where the sines and cosines for the Givens rotations are given by:
\begin{equation}\label{eq:GRJcoefficients}
s_n^{(\alpha,\beta)} = \sqrt{\dfrac{(n+1)(n+\beta+1)}{(n+\alpha+2)(n+\alpha+\beta+2)}},\quad and\quad c_n^{(\alpha,\beta)} = \sqrt{\dfrac{(\alpha+1)(2n+\alpha+\beta+3)}{(n+\alpha+2)(n+\alpha+\beta+2)}}.
\end{equation}
\end{theorem}
\begin{proof}
For a clear exposition, an analogous result in terms of unnormalized Jacobi polynomials will be derived first. This is justified by the relation between normalized and unnormalized connection coefficients:
\[
c_{\ell,n}^{(\alpha,\beta)} = \langle \tilde{P}_\ell^{(\alpha,\beta)}, (1-x)\tilde{P}_n^{(\alpha+2,\beta)} \rangle = \sqrt{\dfrac{\langle P_\ell^{(\alpha,\beta)}, P_\ell^{(\alpha,\beta)}\rangle}{\langle (1-x)P_n^{(\alpha+2,\beta)}, (1-x)P_n^{(\alpha+2,\beta)}\rangle}}\dfrac{\langle P_\ell^{(\alpha,\beta)}, (1-x)P_n^{(\alpha+2,\beta)}\rangle}{\langle P_\ell^{(\alpha,\beta)}, P_\ell^{(\alpha,\beta)}\rangle}.
\]
Using the decrement operator:
\[
(1-x)P_n^{(\alpha+2,\beta)}(x) = \dfrac{2(n+\alpha+2)}{2n+\alpha+\beta+3}P_n^{(\alpha+1,\beta)}(x) - \dfrac{2(n+1)}{2n+\alpha+\beta+3}P_{n+1}^{(\alpha+1,\beta)}(x),
\]
it will suffice it to consider the inner products $\langle P_\ell^{(\alpha,\beta)}, P_n^{(\alpha+1,\beta)} \rangle$. Expanding $P_n^{(\alpha+1,\beta)}$ in the basis of $P_k^{(\alpha,\beta)}$ via Eq.~\eqref{eq:Jacobisingleparamcnxn}, we may conveniently express the inner products as:
\[
\langle P_\ell^{(\alpha,\beta)}, P_n^{(\alpha+1,\beta)} \rangle = \dfrac{(\beta+1)_n}{(\alpha+\beta+2)_n} \dfrac{(\alpha+\beta+1)_\ell(2\ell+\alpha+\beta+1)}{(\beta+1)_\ell(\alpha+\beta+1)} \langle P_\ell^{(\alpha,\beta)}, P_\ell^{(\alpha,\beta)} \rangle.
\]
Thus, for $\ell\le n$, and simplifying using properties of the Pochhammer symbol~\cite[\S 6.1]{Abramowitz-Stegun-65}:
\[
\dfrac{\langle P_\ell^{(\alpha,\beta)}, (1-x)P_n^{(\alpha+2,\beta)} \rangle}{\langle P_\ell^{(\alpha,\beta)}, P_\ell^{(\alpha,\beta)} \rangle} = \dfrac{2(\alpha+1)(\beta+1)_n}{(n+\alpha+\beta+2)(\alpha+\beta+2)_n} \dfrac{(\alpha+\beta+1)_\ell(2\ell+\alpha+\beta+1)}{(\beta+1)_\ell(\alpha+\beta+1)}.
\]
For $\ell = n+1$, 
\[
\langle P_{n+1}^{(\alpha,\beta)}, (1-x)P_n^{(\alpha+2,\beta)} \rangle = \dfrac{2(n+\alpha+2)}{2n+\alpha+\beta+3}\langle P_{n+1}^{(\alpha,\beta)}, P_n^{(\alpha+1,\beta)}\rangle - \dfrac{2(n+1)}{2n+\alpha+\beta+3}\langle P_{n+1}^{(\alpha,\beta)}, P_{n+1}^{(\alpha+1,\beta)}\rangle,
\]
but the first inner product is zero since $\deg(P_n^{(\alpha+1,\beta)}) = n$.
\begin{align*}
\dfrac{\langle P_{n+1}^{(\alpha,\beta)}, (1-x)P_n^{(\alpha+2,\beta)} \rangle}{\langle P_{n+1}^{(\alpha,\beta)}, P_{n+1}^{(\alpha,\beta)} \rangle} & = -\dfrac{2(n+1)}{2n+\alpha+\beta+3}\dfrac{(\beta+1)_{n+1}}{(\alpha+\beta+2)_{n+1}} \dfrac{(\alpha+\beta+1)_{n+1}(2n+\alpha+\beta+3)}{(\beta+1)_{n+1}(\alpha+\beta+1)},\\
& = -\dfrac{2(n+1)}{n+\alpha+\beta+2}.
\end{align*}
Eq.~\eqref{eq:Jacobicoefficients} is then a restatement of the results in terms of orthonormalized Jacobi polynomials.

To determine the Givens rotations, start by applying a Givens rotation from the left to introduce a zero in the second row of the first column. Since the columns of $C^{(\alpha,\beta)}$ are orthonormal, $s_0^{(\alpha,\beta)} = -c_{1,0}^{(\alpha,\beta)}$. Apply another Givens rotation from the left to introduce a zero in the third row of the second column of the conversion matrix. Again, we find that $s_1^{(\alpha,\beta)} = -c_{2,1}^{(\alpha,\beta)}$. Due to the orthonormality, the first rotation introduces zeros in every entry of the first row but the first. Similarly, the second rotation introduces zeros in every entry of the second row but the second. Continuing with $n-2$ more Givens rotations, we arrive at $I_{(n+2)\times(n+1)}$.
\end{proof}

\begin{remark}
A part of Theorem~\ref{theorem:Jacobi} is essentially proved by Olver and Xu~\cite[Lemma 3.1]{Olver-Xu-17} in a different context and for a different purpose. We note, however, that the interpretation of matrices of connection coefficients as generator-representable subdiagonal-plus-semiseparable matrices with orthonormal columns and the analytical representation of the Givens rotations is first described by Slevinsky~\cite[Theorem 2.4]{Slevinsky-ACHA-17}. The present case is recorded for normalized Jacobi polynomials.
\end{remark}

Due to the symmetry relation, a similar result is valid when the second parameter is decremented.

\begin{corollary}
The connection coefficients between $(1+x)\tilde{P}_n^{(\alpha,\beta+2)}(x)$ and $\tilde{P}_{\ell}^{(\alpha,\beta)}(x)$ are:
\begin{equation}
c_{\ell,n}^{(\alpha,\beta)} = \left\{\begin{array}{ccc} (-1)^{n-\ell}(\beta+1)\sqrt{\dfrac{\dfrac{(2\ell+\alpha+\beta+1)\Gamma(\ell+\alpha+\beta+1)\Gamma(\ell+\beta+1)}{\Gamma(\ell+\alpha+1)\Gamma(\ell+1)}}{\dfrac{\Gamma(n+\alpha+\beta+3)\Gamma(n+\beta+3)}{(2n+\alpha+\beta+3)\Gamma(n+\alpha+1)\Gamma(n+1)}}}, & \for & \ell \le n,\\
\sqrt{\dfrac{(n+1)(n+\alpha+1)}{(n+\beta+2)(n+\alpha+\beta+2)}}, & \for & \ell = n+1,\\
0, & & otherwise.
\end{array} \right.
\end{equation}
Furthermore, the matrix of connection coefficients has the same representation via the product of $n$ Givens rotations, where the sines and cosines for the Givens rotations are now given by:
\begin{equation}
s_n^{(\alpha,\beta)} = -\sqrt{\dfrac{(n+1)(n+\alpha+1)}{(n+\beta+2)(n+\alpha+\beta+2)}},\quad and\quad c_n^{(\alpha,\beta)} = \sqrt{\dfrac{(\beta+1)(2n+\alpha+\beta+3)}{(n+\beta+2)(n+\alpha+\beta+2)}}.
\end{equation}
\end{corollary}

\subsection{The weighted Jacobi differential equation}

It is well-known that Jacobi polynomials satisfy the second-order linear homogeneous differential equation~\cite[\S 22.6]{Abramowitz-Stegun-65}:
\begin{equation}\label{eq:SLP_JP}
-\dfrac{{\rm d}}{{\rm d}x}\left[(1-x)^{\alpha+1}(1+x)^{\beta+1}\dfrac{{\rm d}}{{\rm d}x}\tilde{P}_n^{(\alpha,\beta)}(x)\right] = n(n+\alpha+\beta+1)(1-x)^\alpha(1+x)^\beta\tilde{P}_n^{(\alpha,\beta)}(x).
\end{equation}
This differential equation appears in several other forms; however, not one of the forms is useful in considering the Sturm--Liouville problem for weighted Jacobi polynomials that are orthonormal functions in the Hilbert space $L^2([-1,1],\ud x)$:
\begin{equation}\label{eq:WNJP}
\hat{P}_n^{(\alpha,\beta)}(x) := (1-x)^{\frac{\alpha}{2}}(1+x)^{\frac{\beta}{2}}\tilde{P}_n^{(\alpha,\beta)}(x).
\end{equation}
We will now state the differential equation for the weighted normalized Jacobi polynomials.
\begin{theorem}
The weighted normalized Jacobi polynomials $\hat{P}_n^{(\alpha,\beta)}(x)$ are eigenfunctions of the linear differential equation:
\begin{align}\label{eq:SLP_WJP}
-(1-x^2)\dfrac{{\rm d}}{{\rm d}x}\left[(1-x^2)\dfrac{{\rm d}}{{\rm d}x}\hat{P}_n^{(\alpha,\beta)}(x)\right] + & \left((\tfrac{\alpha}{2})^2(1+x)^2+(\tfrac{\beta}{2})^2(1-x)^2-\tfrac{(\alpha\beta+\alpha+\beta)}{2}(1-x^2)\right)\hat{P}_n^{(\alpha,\beta)}(x)\nonumber\\
& = n(n+\alpha+\beta+1)(1-x^2)\hat{P}_n^{(\alpha,\beta)}(x).
\end{align}
\end{theorem}
\begin{proof}
Eq.~\eqref{eq:SLP_WJP} follows directly from Eqs.~\eqref{eq:SLP_JP} and \eqref{eq:WNJP}.
\end{proof}
This Sturm--Liouville problem is the two parameter generalization of Eq.~\eqref{eq:SLP_ALF} for associated Legendre functions. When we view multiplication by $1\pm x$ as an operator acting on the basis of weighted normalized Jacobi polynomials, we immediately identify that certain terms in Eq.~\eqref{eq:SLP_WJP} are symmetric positive-definite and banded, in fact pentadiagonal.

If we expand the weighted normalized Jacobi polynomials of parameters $\gamma$ and $\delta$ in the basis of weighted normalized Jacobi polynomials of parameters $\alpha$ and $\beta$:
\[
\hat{P}_n^{(\gamma,\delta)}(x) = \sum_\ell c_{\ell,n}^{(\alpha,\beta,\gamma,\delta)} \hat{P}_\ell^{(\alpha,\beta)}(x),
\]
then we may rewrite the Sturm--Liouville problem in Eq.~\eqref{eq:SLP_WJP} symbolically:
\begin{equation}\label{eq:JHT}
\left(\MM \DD + \left[(\tfrac{\gamma}{2})^2-(\tfrac{\alpha}{2})^2\right]\MM^+ + \left[(\tfrac{\delta}{2})^2-(\tfrac{\beta}{2})^2\right]\MM^--\left(\tfrac{\gamma\delta+\gamma+\delta-\alpha\beta-\alpha-\beta}{2}\right)\MM\right)u = \lambda \MM u.
\end{equation}
Now, $\MM$ represents multiplication by $1-x^2$, $\MM^+$ is multiplication by $(1+x)^2$, $\MM^-$ is multiplication by $(1-x)^2$, and $\DD$ is the diagonal scaling of the basis by $n(n+\alpha+\beta+1)$ for $n\ge0$.

Due to the recurrence relation:
\begin{align*}
x\tilde{P}_n^{(\alpha,\beta)} & = 2\sqrt{\dfrac{(n+1)(n+\alpha+1)(n+\beta+1)(n+\alpha+\beta+1)}{(2n+\alpha+\beta+1)(2n+\alpha+\beta+2)^2(2n+\alpha+\beta+3)}}\tilde{P}_{n+1}^{(\alpha,\beta)}\\
& + \dfrac{(\beta^2-\alpha^2)}{(2n+\alpha+\beta)(2n+\alpha+\beta+2)}\tilde{P}_n^{(\alpha,\beta)} + 2\sqrt{\dfrac{n(n+\alpha)(n+\beta)(n+\alpha+\beta)}{(2n+\alpha+\beta-1)(2n+\alpha+\beta)^2(2n+\alpha+\beta+1)}}\tilde{P}_{n-1}^{(\alpha,\beta)},
\end{align*}
all the multiplication operators may be derived analytically and algorithms for their component-wise computation to high relative accuracy are described in Appendix~\ref{appendix:JHT}.

If we let $\mathcal{S}$ denote the weighted sum of symmetric positive-definite multiplication operators on the left-hand side of Eq.~\eqref{eq:JHT}, we have the problem:
\[
\left( \MM\DD + \mathcal{S} \right)u = \lambda \MM u,
\]
where $\MM$ is symmetric positive-definite and banded, $\DD$ is diagonal, and $\mathcal{S}$ is symmetric and banded. If we use the same strategy as for spherical harmonics, we would take the Cholesky factorization of $\MM = \RR^\top\RR$ and rearrange to:
\begin{equation}\label{eq:SDBJHT}
\left( \RR\DD\RR^\top + \RR^{-\top}\mathcal{S}\RR^\top \right)v = \lambda \RR\RR^\top v.
\end{equation}
As before, $\RR\DD\RR^\top$ is symmetric and banded and $\RR\RR^\top$ is symmetric positive-definite and banded, but the new term $\RR^{-\top}\mathcal{S}\RR^\top$ is no longer proportional to the identity as occurs for spherical harmonics. In fact, not much of its structure is apparent by this formulation alone.

To reveal the structure of $\RR^{-\top}\mathcal{S}\RR^\top$, we make use of another property of our eigenproblem. Since $\MM$ and $\mathcal{S}$ are multiplication operators with the same separable Hilbert spaces attached to the domain and range, they commute:
\[
[\MM,\mathcal{S}] = 0.
\]
We use the commutator and the (formal) invertibility of $\MM$ to write $\mathcal{S} = \MM\mathcal{S}\MM^{-1}$, leading to the equivalent representation:
\[
\left( \RR\DD\RR^\top + \RR\mathcal{S}\RR^{-1} \right)v = \lambda \RR\RR^\top v.
\]
By the Cholesky factorization, $\RR$ is upper triangular and banded. Thus $\RR^{-\top}\mathcal{S}\RR^\top$ has nonzero entries up to and including the second superdiagonal, but no higher. Furthermore:
\[
\RR\mathcal{S}\RR^{-1} = \left(\RR^{-\top}\mathcal{S}^\top\RR^\top\right)^\top = \left(\RR^{-\top}\mathcal{S}\RR^\top\right)^\top = \RR^{-\top}\mathcal{S}\RR^\top.
\]
While $\RR^{-\top}\mathcal{S}\RR^\top$ in Eq.~\eqref{eq:SDBJHT} initially appears to be asymmetric and dense, it is in fact symmetric and banded.

\section{Discussion}

This work motivates the development and analysis of fast divide-and-conquer algorithms for the symmetric-definite banded generalized eigenvalue problem and the symmetric diagonal-plus-generator-representable eigenvalue problem. While the pre-computations are effectively conquered by the algorithms proposed above, it is anticipated that they may require extended precision arithmetic to ensure that high relative accuracy is guaranteed in the structured representations of the transforms, analogous to the requirements in~\cite{Rokhlin-Tygert-27-1903-06,Tygert-227-4260-08}. While symmetric arrowhead eigensolvers are backward stable~\cite{Borges-Gragg-11-93,Gu-Eisenstat-16-172-95}, providing exact solutions to symmetric arrowhead matrices nearby in $2$-norm, the quadratic spacing of the eigenvalues of the Sturm--Liouville problems dictates that such matrices may not be nearby at all in practice. The use of extended precision arithmetic would only scale the pre-computations by a constant factor that depends on the software or hardware implementation. It is also anticipated that eigensolvers may be facilitated by complete knowledge of the spectra of the differential equations.

The symmetrization of the banded eigenproblems while preserving the bandwidth comes at the expense of reformulation in terms of symmetric-definite generalized eigenvalue problems. The recovery of the orthonormal eigenfunctions now depends on the conditioning of the Cholesky factorization of the multiplication of $1-x^2$, which appears to be of order $\OO(n)$. While extended precision arithmetic should help alleviate some numerical difficulties imposed by the ill-conditioning, the difficulties cannot be eliminated completely. This is arguably the single most important factor for considering the symmetric diagonal-plus-generator-representable eigenvalue problem.

The software package {\tt FastTransforms.jl}~\cite{Slevinsky-GitHub-FastTransforms} implements the fast and backward stable transforms between spherical harmonic expansions and their bivariate Fourier series described in~\cite{Slevinsky-ACHA-17}. As an open source repository with the ability to create light wrappers of BLAS and LAPACK drivers and computational routines and to template in extended precision arithmetic, it is a natural home for the fast transforms described in this work.

\section*{Acknowledgments}

I acknowledge the generous support of the Natural Sciences and Engineering Research Council of Canada through a discovery grant (RGPIN-2017-05514).

\bibliography{/Users/Mikael/Bibliography/Mik}

\appendix

\section{The symmetric-definite banded spherical harmonic generalized eigenvalue problem}\label{appendix:SHT}

\begin{theorem}
\begin{enumerate}
\item The Cholesky factor $\RR$ is:
\[
\RR = \begin{pmatrix} c_1 & 0 & d_1\\ & c_2 & & d_2\\ & & c_3 & & d_3\\ & & & \ddots & & \ddots\\ & & & & c_n & & d_n\\ & & & & & \ddots & & \ddots\end{pmatrix},
\]
where:
\[
c_n = \sqrt{\dfrac{(n+2m)(n+2m+1)}{(2n+2m-1)(2n+2m+1)}},\quad{\rm and}\quad d_n = -\sqrt{\dfrac{n(n+1)}{(2n+2m+1)(2n+2m+3)}}.
\]
\item The product $\RR\RR^\top$ is also known in closed form:
\[
\RR\RR^\top = \begin{pmatrix} e_1 & 0 & f_1\\ 0 & e_2 & & f_2\\ f_1 & & e_3 & & f_3\\ & \ddots & & \ddots & & \ddots\\ & & f_{n-2} & & e_n & & f_n\\ & & & \ddots & & \ddots & & \ddots\end{pmatrix},
\]
where:
\[
e_n = \frac{2(2m^2+(2n+3)m+n(n+1))}{(2n+2m-1)(2n+2m+3)},\quad{\rm and}\quad f_n = -\sqrt{\frac{n(n+1)(n+2m+2)(n+2m+3)}{(2n+2m+1)(2n+2m+3)^2(2n+2m+5)}},
\]
\item and so is $\RR\DD\RR^\top$:
\[
\RR\DD\RR^\top = \begin{pmatrix} g_1 & 0 & h_1\\ 0 & g_2 & & h_2\\ h_1 & & g_3 & & h_3\\ & \ddots & & \ddots & & \ddots\\ & & h_{n-2} & & g_n & & h_n\\ & & & \ddots & & \ddots & & \ddots\end{pmatrix},
\]
where:
\[
g_n = \frac{4m^4+(12n+2)m^3+(14n^2+6n-6)m^2+(8n^3+8n^2-4n)m+2n(n+1)(n^2+n-1)}{(2n+2m-1)(2n+2m+3)},
\]
and:
\[
h_n = -(n+m+1)(n+m+2)\sqrt{\frac{n(n+1)(n+2m+2)(n+2m+3)}{(2n+2m+1)(2n+2m+3)^2(2n+2m+5)}}.
\]
\end{enumerate}
\end{theorem}
\begin{proof}
\begin{enumerate}
\item This follows from the relations:
\begin{align*}
c_1 & = \sqrt{a_1},\\
d_1 & = \frac{b_1}{c_1},\\
c_2 & = \sqrt{a_2},\\
d_2 & = \frac{b_2}{c_2},\quad\hbox{and for}\quad n\ge 3,\\
c_n & = \sqrt{a_n-d_{n-2}^2},\\
d_n & = \frac{b_n}{c_n}.
\end{align*}
\item With $c_n$ and $d_n$, this follows from the relations:
\begin{align*}
e_n & = c_n^2+d_n^2,\\
f_n & = d_nc_{n+2}.
\end{align*}
\item Similarly, the requisite relations are:
\begin{align*}
g_n & = (m+n-1)(m+n)c_n^2+(m+n+1)(m+n+2)d_n^2,\\
h_n & = (m+n+1)(m+n+2)d_nc_{n+2},
\end{align*}
since the diagonal operator is:
\[
\DD = \begin{pmatrix} m(m+1)\\ & (m+1)(m+2)\\ & & (m+2)(m+3)\\ & & & \ddots\\& & & & (m+n-1)(m+n)\\& & & & & \ddots\end{pmatrix}.
\]
\end{enumerate}
\end{proof}

\section{The symmetric-definite banded weighted Jacobi generalized eigenvalue problem}\label{appendix:JHT}

\begin{theorem}
\begin{enumerate}
\item The multiplication of $1+x$ is:
\[
\MM_1 = \begin{pmatrix} a_1 & b_1\\ b_1 & a_2 & b_2\\ & b_2 & a_3 & b_3\\ & & \ddots & \ddots & \ddots\\ & & & b_{n-1} & a_n & b_n\\ & & & & \ddots & \ddots & \ddots\end{pmatrix},
\]
where:
\[
a_n = \dfrac{2n(2n-2)+(4n+2\beta-2)(\alpha+\beta)}{(2n+\alpha+\beta)(2n+\alpha+\beta-2)},\quad{\rm and}\quad b_n = 2\sqrt{\dfrac{n(n+\alpha)(n+\beta)(n+\alpha+\beta)}{(2n+\alpha+\beta-1)(2n+\alpha+\beta)^2(2n+\alpha+\beta+1)}}.
\]
\item The multiplication of $1-x$ is:
\[
\MM_2 = \begin{pmatrix} c_1 & d_1\\ d_1 & c_2 & d_2\\ & d_2 & c_3 & d_3\\ & & \ddots & \ddots & \ddots\\ & & & d_{n-1} & c_n & d_n\\ & & & & \ddots & \ddots & \ddots\end{pmatrix},
\]
where:
\[
c_n = \dfrac{2n(2n-2)+(4n+2\alpha-2)(\alpha+\beta)}{(2n+\alpha+\beta)(2n+\alpha+\beta-2)},\quad{\rm and}\quad d_n = -2\sqrt{\dfrac{n(n+\alpha)(n+\beta)(n+\alpha+\beta)}{(2n+\alpha+\beta-1)(2n+\alpha+\beta)^2(2n+\alpha+\beta+1)}}.
\]
\item Moreover, the multiplication operators $\MM = \MM_1\MM_2$, $\MM^+ = \MM_1^2$, and $\MM^- = \MM_2^2$ can all be computed to high component-wise relative accuracy by symmetric tridiagonal operator multiplication. Similarly, the multiplication operator $\mathcal{S}$ is the sum of three symmetric pentadiagonal operators.

\item  The entries of the Cholesky factor of $\MM = \RR^\top\RR$ satisfy the pentadiagonal Cholesky recurrence and the products $\RR\RR^\top$ and $\RR\DD\RR^\top$ are the products of banded matrices. The Cholesky factor is particularly nice:
\[
\RR = \begin{pmatrix} e_1 & f_1 & g_1\\ & e_2 & f_2 & g_2\\ & & e_3 & f_3 & g_3\\ & & & \ddots & \ddots & \ddots\\ & & & & e_n & f_n & g_n\\ & & & & & \ddots & \ddots & \ddots\end{pmatrix},
\]
where:
\[
e_n = 2\sqrt{\dfrac{(n+\alpha)(n+\beta)(n+\alpha+\beta)(n+\alpha+\beta+1)}{(2n+\alpha+\beta-1)(2n+\alpha+\beta)^2(2n+\alpha+\beta+1)}},
\]
and:
\[
f_n = \dfrac{2(\alpha-\beta)\sqrt{n(n+\alpha+\beta+1)}}{(2n+\alpha+\beta)(2n+\alpha+\beta+2)},
\]
and:
\[
g_n = -2\sqrt{\dfrac{n(n+1)(n+\alpha+1)(n+\beta+1)}{(2n+\alpha+\beta+1)(2n+\alpha+\beta+2)^2(2n+\alpha+\beta+3)}}.
\]
The inverse Cholesky factor is upper triangular and dense (but still rank-structured):
\[
\RR^{-1} = \begin{pmatrix} R_{1,1}^{-1} & R_{1,2}^{-1} & R_{1,3}^{-1} & \cdots\\ & R_{2,2}^{-1} & R_{2,3}^{-1} & \ddots\\ & & R_{3,3}^{-1} & \ddots\\ & & & \ddots
\end{pmatrix},
\]
and entries generically satisfy:
\[
e_i R_{i,i+j}^{-1} + f_i R_{i+1,i+j}^{-1} + g_i R_{i+2,i+j}^{-1} = \delta_{i,j},
\]
where $\delta_{i,j}$ is the Kronecker delta function~\cite[Chap.~24]{Abramowitz-Stegun-65}. In particular:
\[
R_{n,n}^{-1} = \frac{1}{2}\sqrt{\dfrac{(2n+\alpha+\beta-1)(2n+\alpha+\beta)^2(2n+\alpha+\beta+1)}{(n+\alpha)(n+\beta)(n+\alpha+\beta)(n+\alpha+\beta+1)}},
\]
and:
\[
R_{n,n+1}^{-1} = \dfrac{(\beta-\alpha)}{2}\sqrt{\dfrac{n(2n+\alpha+\beta-1)(2n+\alpha+\beta+1)^2(2n+\alpha+\beta+3)}{(n+\alpha)_2(n+\beta)_2(n+\alpha+\beta)_3}},
\]
and:
\begin{align*}
R_{n,n+2}^{-1} & = \dfrac{1}{8}\sqrt{\dfrac{(n)_2(2n+\alpha+\beta-1)(2n+\alpha+\beta+2)^2(2n+\alpha+\beta+5)}{(n+\alpha)_3(n+\beta)_3(n+\alpha+\beta)_4}}\\
& \times \Big((2n+\alpha+\beta)(2n+\alpha+\beta+4)+3(\alpha-\beta)^2\Big).
\end{align*}
\item Finally, to compute the product $\RR^{-\top}\mathcal{S}\RR^\top = \RR\mathcal{S}\RR^{-1}$, we rely on the fact that the result is symmetric and pentadiagonal. This implies that only the main diagonal and the first two superdiagonals of $\RR^{-1}$ are required.
\end{enumerate}
\end{theorem}

\end{document}